\documentclass[12pt]{amsart}
\usepackage{amsmath, amsthm, amssymb}
\usepackage{fullpage}
\usepackage{color}
\usepackage{hyperref}

\newtheorem{theorem}{Theorem}
\newtheorem{lemma}{Lemma}
\newtheorem{proposition}[lemma]{Proposition}

\numberwithin{lemma}{section}

\newcommand{\Ez}{E_0}
\newcommand{\E}{E}

\newcommand{\Elint}{E^{(3)}_{lin}}

\newcommand{\Ent}{E^{n,(3)}}

\numberwithin{equation}{section}

\newcommand{\R}{{\mathbb R}}

\newcommand{\tW}{{\tilde W}}
\newcommand{\tQ}{{\tilde Q}}

\renewcommand{\H}{{\mathcal H }}

\newcommand{\tG}{\tilde{G}}
\newcommand{\tK}{\tilde{K}}

\newcommand{\tg}{{\tilde g}}
\newcommand{\tk}{{\tilde k}}

\newcommand{\W}{{\mathbf W}}

\newcommand{\tU}{\tilde{U}}
\newcommand{\tY}{\tilde{Y}}

\begin{document}

\title{The NLS approximation for two dimensional deep gravity waves}

\author{Mihaela Ifrim}
 \address{Department of Mathematics, University of Wisconsin at Madison}.
\email{ifrim@math.berkeley.edu}
\author{ Daniel Tataru}
\address{Department of Mathematics, University of California at Berkeley}
\email{tataru@math.berkeley.edu}

\begin{abstract}
  This article is concerned with infinite depth gravity water waves in
  two space dimensions.  We consider this system expressed in
  position-velocity potential holomorphic coordinates.  Our goal is to
  study this problem with small wave packet data, and to show that
  this is well approximated by the cubic nonlinear Schr\"odinger equation (NLS) on the natural
  cubic time scale.
\end{abstract}

\maketitle

\section{Introduction}
We consider the two dimensional water wave equations with infinite
depth and with gravity, but without surface tension.  This is governed by
the incompressible Euler's equations in a moving domain $\Omega_t$
with boundary conditions on the water surface $\Gamma_t$, which is viewed 
as a free boundary.

Under the additional assumption that the flow is irrotational, the
fluid dynamics can be expressed in terms of a one-dimensional
evolution of the water surface coupled with the trace of the velocity
potential on the surface. 

The linearization of the water wave equations around the zero solution 
is a dispersive flow, with the dispersion relation $\tau = \pm \sqrt{|\xi|}$.
The NLS approximation corresponds to solutions which are localized
on one of these two branches near a fixed frequency $\xi_0$ on the frequency 
scale $\delta \xi = \epsilon \ll 1$.  Considering the quadratic  approximation 
for the dispersion relation, one sees that the solutions to the linearized equation
are well approximated by the appropriate linear Schr\"odinger equation
up to a quartic $(\epsilon^{-3})$ time scale. 

To the above linear approximation one adds nonlinear effects. Quadratic interactions 
are non-resonant, but the cubic ones are not. Because of this, for solutions 
of amplitude $\epsilon$, one begins to see nonlinear (cubic) effects at cubic 
time scale $\epsilon^{-2}$.  To capture these effects, one naturally replaces 
the  linear Schr\"odinger equation with a cubic NLS problem. A formal derivation of this approximation
was  first obtained  by Zakharov~\cite{zak} for infinite depth and Hasimoto and 
Ono \cite{ho} for finite depth. For simpler dispersive models such an approximation 
has been rigorously justified in \cite{MR2372806,MR2898885,MR3310504,2016arXiv160508704D}.

Our main result asserts that, indeed, on the cubic time scale $|t| < T \epsilon^{-2}$ the 
solutions to the water wave equation with wave packet data are well approximated
by the appropriate cubic NLS flow. Here $T$ can be chosen arbitrarily large,
and represents the effective NLS time.  Our result simplifies and improves 
 earlier work in \cite{MR2800153,MR2891875}. See also \cite{MR3461357} for finite depth, as well as the survey article  \cite{MR3409892}.

\subsection{ The incompressible Euler equations}

We consider the incompressible, infinite depth water wave equation in
two space dimensions with gravity but no surface tension. This is
governed by the incompressible Euler's equations with boundary
conditions on the water surface.

We briefly provide the Eulerian formulation of the equations. We
denote the water domain at time $t$ by $\Omega(t)$, and the water
surface at time $t$ by $\Gamma(t)$. We think of $\Gamma(t)$ as being
either asymptotically flat at infinity. The fluid velocity is denoted
by $u$ and the pressure is $p$. Then $u$ solves the Euler's equations
inside $\Omega(t)$,
\begin{equation}
\left\{
\begin{aligned}
& u_t + u \cdot \nabla u = -\nabla p  -  g j
\\
& \text{div } u = 0
\\
& u(0,x) = u_0(x),
\end{aligned}
\right.
\end{equation}
while on the boundary we have the dynamic boundary condition
\begin{equation}
 p = 0  \ \ \text{ on } \Gamma (t ),
\end{equation}
and the kinematic boundary condition 
\begin{equation}
\partial_t+ u \cdot \nabla \text{ is tangent to } \bigcup \Gamma (t),
\end{equation}
where $g$ represents the gravity.

Under the additional assumption that the flow is irrotational, we can
write $u$ in terms of a velocity potential $\phi$ as $u = \nabla \phi$, 
where $\phi$ is harmonic within the fluid domain, with appropriate decay at infinity.
Thus $\phi$ is determined by its trace on the free boundary
$\Gamma (t)$. Denote by $\eta$ the height of the water surface as a
function of the horizontal coordinate. Following Zakharov, we introduce 
$\psi =\psi (t,x)\in \mathbb{R}$ to be the trace of the velocity potential $\phi$ on the boundary,
$\psi (t,x)=\phi (t,x,\eta(t,x))$. Then the fluid dynamics can be expressed in terms of a 
one-dimensional evolution of  the pairs of variables
$(\eta,\psi)$, namely 
\begin{equation}
\left\{
\begin{aligned}
& \partial_t \eta - G(\eta) \psi = 0 \\
& \partial_t \psi + g \eta  +\frac12 |\nabla \psi|^2 - \frac12 
\frac{(\nabla \eta \cdot \nabla \psi +G(\eta) \psi)^2}{1 +|\nabla \eta|^2} = 0 .
\end{aligned}
\right.
\end{equation}
Here $G$ represents the Dirichlet to Neumann map associated to the fluid domain.
This is Zakharov's  Eulerian formulation of the gravity water wave equations.

It is known since Zakharov~\cite{zak} that the water-wave 
system is Hamiltonian, where the Hamiltonian (conserved energy) is given by
\[
\mathcal{H}(\eta ,\psi)= \frac{g}{2} \int_{\mathbb{R}}\eta^2\, dx +\frac{1}{2}\int_{\mathbb{R}}\int_{-\infty}^{\eta(x)} \vert \nabla_{x,y}\phi\vert^2 \, dy \, dx.
\]

The above formulation is only given here for a complete picture of the way
the two-dimensional water wave equations were traditionally
described. Besides the Eulerian setting, there is also the Lagrangian
formulation which we will not describe because it is not in any
way related with the goals of our paper. 

Both the Eulerian and Lagrangian formulations work in any spatial dimension. However, in
this paper we are only interested in the two dimensional case,
where we have available an additional choice of coordinates, 
namely the  \emph{holomorphic  coordinates} (or conformal coordinates).
We found these coordinates very useful and easy to use for a number 
of problems  in the two dimensional water waves realm (see for
example \cite{HIT, ITCapillary, IT-global, ITVorticity}), and we will
also prefer them here.

\subsection{ Water waves in holomorphic coordinates}

Here we will work with the water wave equations in holomorphic
coordinates, defined via a conformal (holomorphic) parametrization of
the fluid domain by the lower half-space $\H = \{ \Im z < 0\}$ with
coordinates denoted by $z = \alpha+i\beta$.  The conformal map
\[
Z: \H \to \Omega_t
\]
is chosen to map $\R$ (the boundary of $\H$) into the free boundary $\Gamma_t$, and is uniquely 
determined by the requirement that
\[
\lim_{\H \ni z \to \infty} Z(z) - z = 0.
\]
Removing the leading part 
\[
W(z) := Z(z) - z,
\]
we obtain our first dynamic variable $W$ which describes the
parametrization of the free boundary.

Our second dynamic variable, the holomorphic velocity potential,
denoted by $Q$, is represented in terms of the real velocity potential
$\phi$ and its harmonic conjugate $q$ (the stream function) as
\[
Q = \phi+iq.
\]

Expressed in holomorphic position/velocity potential 
coordinates $(W,Q)$ the water wave equations 
have the form 
\begin{equation}
\label{ww2d1}
\left\{
\begin{aligned}
& W_t + F (1+W_\alpha) = 0, \\
& Q_t + F Q_\alpha -i W + P\left[ \frac{|Q_\alpha|^2}{J}\right]  = 0, \\
\end{aligned}
\right.
\end{equation}
where $P$ is the projection onto negative frequencies, alternatively defined as
\[
P:=\frac{1}{2}\left(I-iH\right),
\]
wit$H$ denoting the Hilbert transform,  
and $F$ is given by
\begin{equation}
\label{rww}
 F := P\left[ \frac{Q_\alpha - \bar Q_\alpha}{J}\right] , \qquad J = |Z_\alpha|^2.
\end{equation}

This equations are interpreted as an evolution in the space of
holomorphic functions, where, by a slight abuse of terminology, we
call a function on the real line holomorphic if it admits a bounded
holomorphic extension to the lower half space $\H$.

For a complete derivation of the above equations we refer the reader
to \cite{HIT}.  The use of such coordinates was pioneered by
Ovsiannikov~\cite{ov}, and further developed by Wu~\cite{wu} and
Dyachenko-Kuznetsov-Spector-Zakharov~\cite{zakharov}. Our notations
here follow \cite{HIT} and are closer to the system formulation in
\cite{zakharov}, except that we prefer to work with complex valued
holomorphic functions as in \cite{wu} rather than with their real
parts as in \cite{zakharov}. This is in part due to the algebra structure
available on the space of holomorphic functions.

The system \eqref{ww2d1} is a Hamiltonian system, with the energy
(Hamiltonian) given by
\[
\E  (W,Q)=  \int \frac12 |W|^2 + \frac1{2i} (Q \bar Q_\alpha - \bar Q Q_\alpha)
- \frac{1}{4} (\bar W^2 W_\alpha + W^2 \bar W_\alpha)\, d\alpha.
\]
This corresponds to the energy space 
\[
\H := L^2 \times \dot H^\frac12.
\]
We will also use the higher regularity spaces 
\[
\H^k = \{ (w,q);  \ \partial^j  (w,q)  \in \H, \ j = 0, \cdots,k\}.
\]

Viewed as a nonlinear hyperbolic system, the water wave system is degenerate 
hyperbolic with a double speed. It is also non-diagonal, so to better understand it
it is useful to diagonalize it and introduce the appropriate counterpart of 
\emph{ Alihnac's good variable}. This is most readily seen at the level 
of the differentiated equation (see \cite{HIT}), where the good variable is given by 
\[
(\W,R):= (W_\alpha, \frac{Q_\alpha}{1+W_\alpha}).
\]
Here the new holomorphic variable $R$ represents exactly the Eulerian velocity field
 in complex notation and with the holomorphic parametrization.

 Using these variables one can convert the system \eqref{ww2d1} into a
 quasilinear system for $(\W,R)$ by differentiating it.  This has the
 form
\begin{equation} \label{ww2d-diff}
\left\{
\begin{aligned}
 & \W_{ t} + b \W_{ \alpha} + \frac{(1+\W) R_\alpha}{1+\bar \W}   =  (1+\W)M
\\
& R_t + bR_\alpha = i\left(\frac{\W - a}{1+\W}\right),
\end{aligned}
\right.
\end{equation}
where the advection coefficient (i.e. the double speed, or the
Eulerian velocity field restricted to the surface in the holomorphic
parametrization) is
\begin{equation}
b := \Re F = P \left[\frac{{Q}_\alpha}{J}\right] +  \bar P\left[\frac{\bar{Q}_\alpha}{J}\right].
\label{defb}
\end{equation}
 The real {\em frequency-shift} $a$ is given by
\begin{equation}
a := i\left(\bar P \left[\bar{R} R_\alpha\right]- P\left[R\bar{R}_\alpha\right]\right),
\label{defa}
\end{equation}
and  the auxiliary function $M$ is given by
\begin{equation}\label{M-def}
M :=  \frac{R_\alpha}{1+\bar \W}  + \frac{\bar R_\alpha}{1+ \W} -  b_\alpha =
\bar P [\bar R Y_\alpha- R_\alpha \bar Y]  + P[R \bar Y_\alpha - \bar R_\alpha Y],
\end{equation}
with
\[
Y := \frac{\W}{1+\W}.
\]
The differentiated  system \eqref{ww2d-diff} is self contained,  
governs an evolution in the space of holomorphic functions, 
and can  be used both  directly and in its projected version (which are equivalent).

As shown in \cite{HIT}, the system \eqref{ww2d-diff} can be studied independently of the original system
\eqref{ww2d1}, and also independently of whether the solution $(\W,R)$ corresponds to a non-self-intersecting 
fluid surface. Precisely, we have the following well-posedness result:

\begin{theorem}[\cite{HIT}] \label{t:lwp}
The system \eqref{ww2d-diff} is locally well-posed in $\H^1$.
\end{theorem}
This is a quasilinear well-posedness result, so it includes existence, uniqueness, continuous 
dependence on the initial data in the strong topology $\H^1$ as well as Lipschitz dependence on the data
in a weaker topology $\H$.

Furthermore, in \cite{HIT} a cubic lifespan result was proved in the same spaces:

\begin{theorem}[\cite{HIT}] \label{t:cubic}
Consider the system \eqref{ww2d-diff} with initial data satisfying 
\[
\| (\W,R)(0)\|_{\H^1} \leq \epsilon.
\]
Then the solution $(W,R)$ exists at least for a lifespan
\[
T_\epsilon = \frac{c}{\epsilon^2}
\]
with uniform bounds,
\begin{equation}
\| (\W,R)(t)\|_{\H^1} \lesssim \epsilon, \qquad |t| \leq T_\epsilon.
\end{equation}
\end{theorem}
This result is accompanied in \cite{HIT} by uniform bounds for the linearized equation 
for the undifferentiated system \eqref{ww2d1} in the
space $\H$, which are heavily used here and are recalled in Section~\ref{s:energy}.

\subsection{The NLS approximation to the water wave system}
In this article we consider a different class of solutions,
which corresponds to the cubic NLS approximation to the 
gravity wave system \eqref{ww2d1}. To describe this approximation
we begin with the linearized evolution around the zero solution,
which has the form
\begin{equation}
\label{ww2d-lin0}
\left\{
\begin{aligned}
& w_t + q_\alpha = 0 \\
& q_t  -i w   = 0, \\
\end{aligned}
\right.
\end{equation}
where $(w,q)$ are restricted to the space of holomorphic functions, i.e., only with negative frequencies.

The dispersion relation for this linear evolution has two branches, namely 
\[
\tau = \omega_{\pm}(\xi):=\pm \sqrt{|\xi|}, \qquad \xi \leq 0,
\]
where $\tau$ stands for the time Fourier variable, and $\xi$ is the
spatial Fourier variable.  The two branches correspond to linear waves
which travel to the right, respectively to the left.  The NLS
approximation applies to solutions which are localized on a single
branch, and furthermore are also localized near a single frequency
$\xi_0$. In view of the scaling symmetry and the time reversal
symmetry, we can without any restriction of generality, set
$\xi_0 = -1$ and choose the ``$+$'' sign.

The quadratic approximation for the ``$+$'' dispersion relation near $\xi_0 = - 1$ is 
\[
\tau =\omega_0(\xi) :=  1 - \frac12 (\xi+1)  - \frac18 (\xi+1)^2, 
\]
which corresponds to the linear evolution 
\begin{equation}
(i \partial_t + \omega_0(D)) y  = 0.
\end{equation}
This can be recast as the linear Schr\"odinger flow 
\begin{equation}
\left(i \partial_t + \frac18 \partial_x^2\right) u = 0
\end{equation}
via the transformation
\[
y(t,x) = e^{i t} e^{-i  x} u(t,x  -\frac12 t). 
\]

Returning now to the linearized water wave system \eqref{ww2d-lin0},
solutions near frequency $-1$ on the $+$ branch should be well approximated
by the functions 
\[
(w,q) \approx ( y, y).
\]
For a more quantitative analysis, suppose that we are looking at
solutions concentrated in an $\epsilon$ neighbourhood of $\xi_0 = 1$,
then we have the difference relation
\[
|\omega_0 - \omega_+| \lesssim \epsilon^3.
\]
Thus the  linear evolution operators $e^{it \omega_0(D)}$ are a good approximation for 
$e^{it \omega_+(D)}$ on a time scale 
\[
|t| \ll \epsilon^{-3},
\]
but (the difference) stays  within $O(\epsilon)$ 
only on  a time scale 
\[
|t| \ll \epsilon^{-2}.
\]

Now we consider the nonlinear setting. We begin with a solution $U$ for the NLS equation
\begin{equation}\label{NLS}
(i \partial_t + \frac{1}{8} \partial_x^2) U =  \lambda U |U|^2,
\end{equation}
where $\lambda$ remains to be chosen later. We rescale this to the $\epsilon$ frequency scale,
\begin{equation}
U^\epsilon(t,x) := \epsilon U(\epsilon^2 t,\epsilon x) ,
\end{equation}
which still solves \eqref{NLS}, and then define $Y^\epsilon$ by 
\begin{equation}\label{Ye-def}
Y^\epsilon(t,x) := e^{i t} e^{-i  x} U^{{ \epsilon}}(t,x  -\frac12 t) ,
\end{equation}
which solves an NLS equation with the dispersion relation given by $\omega_0$,
\begin{equation}\label{NLS-y}
(i \partial_t + \omega_0(D)) Y^{\epsilon} =    \lambda Y^\epsilon |Y^\epsilon|^2.
\end{equation}
 Thus $Y^\epsilon$ is localized around frequency $\xi_0 = - 1$ on the $\epsilon$ frequency scale. We remark that 
nonlinear effects are first seen on the unit time scale for $U$, which corresponds to the 
$\epsilon^{-2}$ time scale for $Y$.

Then we ask whether, for a well chosen $\lambda$, 
we can find a solution $(W,Q)$ for the water wave equation \eqref{ww2d1}  
which stays close to $(Y^\epsilon,Y^\epsilon)$ on a timescale
\[
T_\epsilon := T \epsilon^{-2}.
\]
We remark that here it is important that $T$ is independent of $\epsilon$, as on shorter 
time scales  the linear Schr\"oedinger equation is also a good approximation. It is also interesting to allow $T$ 
to be arbitrarily large, as this would show that the water wave equation captures long time NLS dynamics.
This motivates our main result:

\begin{theorem} \label{t:NLS-approx} Let $U_0 \in H^3$, let $U$ be the
  corresponding solution to the cubic NLS equation \eqref{NLS} with
  $\lambda = -\frac12$, and let $Y^\epsilon$ be as in
  \eqref{Ye-def}.  Let $T > 0$. Then there exists
  $0 < \epsilon_0 = \epsilon_0(\|U_0\|_{H^3},T)$ so that for each
  $0 < \epsilon < \epsilon_0$ there exists a solution $(W,Q)$ to
  \eqref{ww2d1} for $t$ in the time interval
  $\{|t| \leq T_\epsilon:= T \epsilon^{-2}\}$ satisfying the following
  estimates:
\begin{equation}\label{error-lo}
\| (W - Y^\epsilon, Q-Y^\epsilon)\|_{\H} \lesssim \epsilon^\frac32,
\end{equation}
\begin{equation}\label{error-hi}
\| (\W - iY^\epsilon, R-iY^\epsilon)\|_{\H^1} \lesssim \epsilon^{\frac32}.
\end{equation}
\end{theorem}
To compare this with the cubic lifespan bound in Theorem~\ref{t:cubic} we 
note that $Y^\epsilon$ (and also $W,\W$ and $R$) satisfies the $L^2$ bound 
\[
\| Y^\epsilon \|_{L^2} \lesssim \epsilon^\frac12,
\]
which is weaker than the $\epsilon$ bound in Theorem~\ref{t:cubic}.

However, here we have the additional bound
\[
\| (D+1) Y^\epsilon\|_{L^2} \lesssim \epsilon^\frac32,
\]
which expresses the localization around frequency $1$. Because of this,
the pointwise bound for $Y^\epsilon$ improves to 
\[
\|Y^\epsilon\|_{L^\infty} \lesssim \epsilon,
\]
which is consistent with the pointwise bounds for the solutions 
 in Theorem~\ref{t:cubic}. This fact is best expressed in terms of 
the control norms $A$ and $B$ introduced in \cite{HIT},  which correspond to uniform  bounds for
  $(\W, R)$ and will be described later in Section~\ref{s:energy}.

  We further remark that the solution $(W,Q)$ obtained in the proof of
  the theorem is actually smooth.  Precisely, we will prove the
  stronger high frequency bound
\begin{equation}\label{error-hihi}
  \| (\W - iY^\epsilon_{\lesssim 1}, R-iY^\epsilon_{\lesssim 1})\|_{\H^{k}} 
\lesssim_k \epsilon^{\frac32}, \qquad k \geq 0.
\end{equation}

Finally, we also show that the bounds for $(W,Q)$ given by Theorem~\ref{t:NLS-approx}
are stable with respect to small $O(\epsilon^{1+})$ perturbations, in a way that is consistent with 
Theorem~\ref{t:cubic}:

\begin{theorem} \label{t:NLS-approx+}
  Let $U_0,U,Y^\epsilon$  be as in Theorem~\ref{t:NLS-approx}. Let $0 < \delta_1 < \delta \leq  \frac12$,
and $(W_0,Q_0)$ be initial data satisfying 
\begin{equation}\label{error-lo+0}
\| (W_0 - Y^\epsilon_0, Q_0-Y^\epsilon_0)\|_{\H} \lesssim \epsilon^{1+\delta},
\end{equation}
\begin{equation}\label{error-hi+0}
\| (\W_0 - iY^\epsilon_0, R_0-iY^\epsilon_0)\|_{\H^1} \lesssim \epsilon^{1+\delta}.
\end{equation}

Then the corresponding solution $(W,Q)$ to the system \eqref{ww2d1} exists  for $t$ in the interval $\{|t| \leq T_\epsilon:= T \epsilon^{-2}\}$ and satisfies the
  following estimates:
\begin{equation}\label{error-lo+}
\| (W - Y^\epsilon, Q-Y^\epsilon)\|_{\H} \lesssim \epsilon^{1+\delta_1},
\end{equation}
\begin{equation}\label{error-hi+}
\| (\W - iY^\epsilon, R-iY^\epsilon)\|_{\H^1} \lesssim \epsilon^{1+\delta_1}.
\end{equation}
\end{theorem}

We have stated this separately from Theorem~\ref{t:NLS-approx} since
its proof has nothing to do with the NLS approximation, once the
solution $(W,Q)$ in Theorem~\ref{t:NLS-approx} is given. This result
likely holds as well with $\delta_1 = \delta = 0$, but then the proof
becomes considerably more technical.

We remark that results similar in spirit have been proved earlier, first in 
  \cite{MR2800153} for a model problem, and then in \cite{MR2891875}
for this problem. Compared with both of these works our approach here 
is considerably  simpler and shorter, and yields stronger results. There 
are three ingredients which enable us to do that, which we hope 
will be useful also in other similar problems:

\begin{itemize}
\item The normal form analysis for the system \eqref{ww2d1}, which is from \cite{HIT}
but with further refinements.

\item The cubic energy estimates also from \cite{HIT}, and more precisely 
the fact that these estimates are based on uniform control norms rather than
Sobolev bounds.

\item The perturbative analysis based on the linearized equation, for
  which we also have cubic bound from \cite{HIT}.
\end{itemize}

The proof of Theorem~\ref{t:NLS-approx} is done in three steps. The first step is to
replace the exact solution $U$ to the NLS equation with a more regular
function  $\tU$, with an $\epsilon$ dependent truncation
scale. The price to pay is that $\tU$ is only an approximate solution to the
cubic NLS problem \eqref{NLS}. Precisely, we define
\begin{equation}
\label{truncation}
\tU := U_{\leq c\epsilon^{-1}} ,
\end{equation}
which solves the equation \eqref{NLS} with an error
\[
\tilde{f} := (i \partial_t + \frac{1}{8} \partial_x^2) \tU-  \lambda \tU |\tU|^2 ,
\]
Here $c$ is a small universal constant. Then we have:

\begin{proposition}\label{p:cut}
Let $s \geq \frac32$. Then  for the function $\tU$ above  we have

\noindent $a)$ The uniform bounds 
\begin{equation}\label{nls-en}
\| \tU\|_{H^k} \lesssim   \epsilon^{s-k}, \qquad k \geq s,
\end{equation}
\noindent $b)$ The difference bounds
\begin{equation}
\| \tU- U\|_{L^2} \lesssim \epsilon^s,
\end{equation}
\noindent $c)$ The error estimates 
\begin{equation}\label{nls-err}
\| \tilde{f}\|_{L^2} \lesssim \epsilon^{s+1}.
\end{equation}
\end{proposition}
\noindent
This result is proved in Section~\ref{s:U-err}.

Corresponding to $\tU$ we define $\tY^\epsilon$ via the definitions
\eqref{Ye-def} and \eqref{truncation}.  The smallness of $c$ above guarantees that
$\tY^{\epsilon}$ is frequency localized in a small neighbourhood of   $\xi_0 = -1$,
\begin{equation}\label{ye-bound} 
 \| (D+1)^s \tY^\epsilon\|_{L^2} \lesssim 1.
\end{equation}
By the previous proposition, $\tY^{\epsilon}$ compares to $Y^\epsilon$ as follows
\begin{equation}\label{ye-bound-diff} 
 \| \tY^\epsilon - Y^\epsilon\|_{H^\frac52} \lesssim \epsilon^{s-\frac52}, \qquad s \geq \frac52.
\end{equation}
In addition, $\tY^\epsilon$ is an approximate solution to \eqref{NLS-y}
\begin{equation}
\label{eqn-ye}
(i \partial_t + \omega_0(D) )  \tY^\epsilon=  \lambda \tY^\epsilon |\tY^\epsilon|^2+g^\epsilon,
\end{equation}
with an error 
\[
g^\epsilon := (i \partial_t + \omega_0(D) )  \tY^\epsilon-  \lambda \tY^\epsilon |\tY^\epsilon|^2
\]
satisfying the bound
\begin{equation}\label{ge-bound}
\| g^\epsilon\|_{L^2} \lesssim \epsilon^{s+\frac72}.
\end{equation}
Later we will use this Proposition with $s = 3$. However, we are
proving it for all $s \geq 1$ in order to emphasize that the
truncation errors are much better than what is needed later on;
thus one sees that this is not the place where our estimates are tight.
The above proposition shows that it suffices to prove
Theorem~\ref{t:NLS-approx} with $Y^{\epsilon}$ replaced by
$\tY^{\epsilon}$.

 The second step is to use normal form analysis in order to construct a good enough 
approximate solution to the water wave equation which is close to $(\tY^{\epsilon},\tY^{\epsilon})$.

\begin{theorem}\label{t:approx}
  Let $\epsilon > 0$. Let $\tY^\epsilon$ be the approximate solution
  defined above to the NLS equation \eqref{NLS-y}. Then there exists a
  frequency localized approximate solution $(W^\epsilon,Q^\epsilon)$
  for the water wave equation \eqref{ww2d1} with properties as
  follows:

(i) Uniform bounds:
\begin{equation}
 \| (W^\epsilon - \tY^\epsilon, Q^\epsilon-\tY^\epsilon)\|_{\H} \lesssim \epsilon^{\frac32}, \qquad 
\| (\W^\epsilon - \tY^\epsilon_\alpha, R^\epsilon -\tY^\epsilon_\alpha)\|_{\H^1} \lesssim  \epsilon^{\frac32}.
\end{equation}

(ii) Approximate solution: 
\begin{equation}
\label{ww2d1-eps}
\left\{
\begin{aligned}
& W^\epsilon_t + F^\epsilon (1+W^\epsilon_\alpha) = g^\epsilon, \\
& Q^\epsilon_t + F^\epsilon Q^\epsilon_\alpha -i W^\epsilon + P\left[ \frac{|Q^\epsilon_\alpha|^2}{J^\epsilon}\right]  = k^\epsilon, \\
\end{aligned}
\right.
\end{equation}
where $(g^\epsilon,k^\epsilon)$ satisfy the bounds
\begin{equation}
 \| (g^\epsilon, k^\epsilon)\|_{\H^{N}} \lesssim \epsilon^{\frac72}.
\end{equation}
\end{theorem}

The third part of the proof is to show that the approximate solution
$(W^\epsilon,Q^\epsilon)$ can be replaced with an exact solution.

\begin{theorem}\label{t:exact}
  Let $T > 0$. Let $(W^\epsilon,Q^\epsilon)$ be the approximate
  solution for the water wave equation \eqref{ww2d1} given by the
  previous theorem. Then for $\epsilon$ small enough (depending only
  on $\|U_0\|_{H^3}$ and on $T$) the exact solution $(W,Q)$ with the
  same initial data exists in the time interval
  $\{|t| \leq T \epsilon^{-2}\}$, and satisfies the bounds
\begin{equation}
\| (W - W^\epsilon, Q-Q^\epsilon)\|_{\H} \lesssim \epsilon^\frac32,
\end{equation}
\begin{equation}
\| (\W - \W^\epsilon, R-R^\epsilon)\|_{\H^1} \lesssim \epsilon^\frac32.
\end{equation}
\end{theorem}

The rest of the paper is organized as follows. In the next section we
review the cubic energy bounds in \cite{HIT} for both the water wave
equation \eqref{ww2d1} and its linearization. In the following section
we perform the normal form analysis leading to the approximate
solution $(\tW,\tQ)$, and prove Theorem~\ref{t:approx}.  In 
Section~\ref{s:exact} we make the transition from the approximate solution to the
exact solution and prove Theorem~\ref{t:exact}. Finally, in the last section we
carry out the perturbative analysis leading to Theorem~\ref{t:NLS-approx+}.

\subsection*{Acknowledgments} The first author
was partially supported by a Clare Boothe Luce Professorship. The second  author was
partially supported by the NSF grant DMS-1800294 as well as by a Simons Investigator
grant from the Simons Foundation.

\section{The NLS truncation}
\label{s:U-err}
Here we prove Proposition~\ref{p:cut}. The uniform $H^k$ bounds for
the NLS problem are well known due to its the conservation laws, so all
we have to do is to prove the error estimates \eqref{nls-err}.
Precisely, we need to estimate the $L^2$ norm of the error
\[
\tilde{f} =  P_{< \epsilon^{-1}} (U |U|^2)  - P_{< \epsilon^{-1}}  U |P_{< \epsilon^{-1}}  U|^2.
\]

To get the estimate we decompose $U$ in frequencies greater or equal,
and frequencies much smaller than $\epsilon^{-1}$ as shown below
\[
U=P_{\geq \epsilon^{-1}}U +P_{\ll \epsilon^{-1}}U,
\]
and then insert this decomposition into the expression of
$\tilde{f}$. Note that the complex conjugate plays no role in the
analysis as it commutes with the spectral projections. The trilinear
expressions arising in $\tilde{f}$ where  all factors are  localized at frequencies
much less than $\epsilon^{-1}$ will cancel. We are left with the following types of terms: 

\begin{description}
\item[(i)] two of the factors
are localized at frequencies higher than $\epsilon^{-1}$, and the
third factor is localized at frequencies much less than
$\epsilon^{-1}$; we denote such terms by $\tilde{f}_{hhl}$. 
\item[(ii)] two
of the factors are localized at frequencies much smaller than
$\epsilon^{-1}$, and the third factor is localized at frequencies
greater than $\epsilon^{-1}$; we denote such terms by
$\tilde{f}_{llh}$; 
\item[(iii)] all three factors are localized at
frequencies higher than $\epsilon^{-1}$; we denote such terms by
$\tilde{f}_{hhh}$. 
\end{description}
We prove the $L^2$ bounds of $\tilde{f}$ by bounding each of the terms
in $f_{llh}$, $f_{hhl}$ and $f_{hhh}$, respectively. 

For completeness we chose one term of the form $f_{hhh}$, and show the $L^2$ bound in
this case. We harmlessly discard the $P_{< \epsilon}$ multipliers:
\[
\begin{aligned}
\Vert P_{\geq \epsilon^{-1}}U\cdot  P_{\geq \epsilon^{-1}}U\cdot  P_{\geq \epsilon^{-1}}\bar{U}  \Vert_{L^2}&\lesssim \Vert P_{\geq \epsilon^{-1}}U \Vert_{L^2}  \Vert P_{\geq \epsilon^{-1}}\bar{U}\Vert_{L^{\infty}} \Vert  P_{\geq \epsilon^{-1}}U\Vert_{L^{\infty}}\\
&\lesssim \epsilon^{s}\cdot \epsilon ^{s-\frac{1}{2}}\cdot \epsilon ^{s-\frac{1}{2}}=\epsilon^{3s-1},
\end{aligned}
\]
which leads to the following bound
\[
\Vert \tilde{f}_{hhh}\Vert_{L^2}\lesssim \epsilon^{3s-1}.
\]

In proving the above bound we have used the uniform bound
\eqref{nls-en} for $L^2$ norm, and Bernstein's inequality for the
remaining two $L^{\infty}$ norms.

We now consider a representative term from the ones that compose
$\tilde{f}_{hhl}$ and get the bound by using \eqref{nls-en} and
Bernstein's inequality, discarding again the $P_{< \epsilon}$ multipliers:
\[
\begin{aligned}
\Vert P_{\geq \epsilon^{-1}}U\cdot  P_{\geq \epsilon^{-1}}U\cdot  P_{\ll \epsilon^{-1}} \bar{U} \Vert_{L^2}&\lesssim
\Vert  P_{\geq \epsilon^{-1}}U\Vert_{L^2}  
\Vert  P_{\geq \epsilon^{-1}}U\Vert_{L^{\infty}} 
\Vert P_{\ll \epsilon^{-1}} \bar{U}\Vert_{L^{\infty}}\\
&\lesssim \epsilon^{s}\cdot \epsilon^{s-\frac{1}{2} }\cdot 1=\epsilon^{2s-\frac{1}{2}}.
\end{aligned}
\]

Lastly, we show the bound for $\tilde{f}_{llh}$.  Here we no longer discard the $P_{< \epsilon}$ multipliers,
in order to take advantage of a commutator structure arising from the difference in $\tilde f$.
Precisely, we need  bounds for following  commutator
\[
\left[ P_{<\epsilon^{-1}}\, , \, P_{\ll \epsilon^{-1}} U\cdot P_{\ll \epsilon^{-1}} U\right] P_{\geq \epsilon^{-1}} \bar{U}.
\]
The commutator structure  shifts a
derivative from the high-frequency function $P_{\geq \epsilon^{-1}} \bar{U}$ to the low-frequency
function $P_{\ll \epsilon^{-1}} U\cdot P_{\ll \epsilon^{-1}} U$, allowing us to obtain the following bound
\[
\Vert \left[ P_{<\epsilon^{-1}}\, , \, P_{\ll \epsilon^{-1}} U\cdot P_{\ll \epsilon^{-1}} U\right] P_{\geq \epsilon^{-1}} \bar{U}.
\Vert_{L^2}\lesssim \epsilon \| \partial_x(P_{\ll \epsilon^{-1}} U\cdot P_{\ll \epsilon^{-1}} U)\|_{L^\infty} \| P_{\geq \epsilon^{-1}} \bar{U} \|_{L^2}
\lesssim \epsilon^{s+1},
\]
which gives the desired bound \eqref{nls-err}.

\section{ Energy estimates}
\label{s:energy}

Here we begin with a brief review of the energy estimates in \cite{HIT}.  For completeness we recall that the nonlinear system \eqref{ww2d1} also admits a conserved energy (the Hamiltonian), which has the form
\begin{equation}\label{ww-energy}
\E(W,Q) = \int \frac12 |W|^2 + \frac1{2i} (Q \bar Q_\alpha - \bar Q Q_\alpha)
- \frac{1}{4} (\bar W^2 W_\alpha + W^2 \bar W_\alpha)\, d\alpha.
\end{equation}
 As suggested by the above energy, our main function spaces for the 
 differentiated water wave system \eqref{ww2d-diff} are the
spaces $\H^k$ endowed with the norm
\[
\| (\W,R) \|_{H^k}^2 := \sum_{n=0}^k
\| \partial^k_\alpha (\W,R)\|_{ L^2 \times \dot H^\frac12}^2,
\]
where $k \geq 1$. 

To describe the lifespan of the solutions we recall the control norms $A$ and $B$ which we
have introduced in \cite{HIT}:
\begin{equation}\label{A-def}
A := \|\W\|_{L^\infty}+ \||D|^\frac12 R\|_{L^\infty \cap B^{0,\infty}_{2}},
\end{equation}
respectively
\begin{equation}\label{B-def}
B :=\||D|^\frac12 \W\|_{BMO} + \| R_\alpha\|_{BMO},
\end{equation}
where $|D|$ represents the multiplier with symbol $|\xi|$.
Here $A$ is a scale invariant quantity, while $B$ corresponds to the
homogeneous $\dot \H^1$ norm of $(\W, R)$. We also note that  both 
$A$ and $B$  are controlled by the $\H^1$ norm of the solution. 

One of the main results in \cite{HIT} was a  cubic energy bound, 
which for convenience we recall below:

\begin{theorem}\label{t:energy} 
For any $n \geq 0$ there exists an energy functional $E^n$ 
which has the following properties as long as $A \ll 1$:

(i) Norm equivalence:
\begin{equation*}
E^n (\W,R)= (1+ O(A)) \| (\W,R)\|_{\H^n}^2
\end{equation*}

(ii) Cubic energy estimates for solutions to \eqref{ww2d-diff}:
\begin{equation*}
\frac{d}{dt} \Ent (\W,R)  \lesssim AB \| (\W,R)\|_{\H^n}^2
\end{equation*}
\end{theorem}

This is a small data result. The cubic energy functional $E^n$ exists
at any level of regularity $n$ of the solutions $(\W,R)$. These energy
functionals are obtained in \cite{HIT} using the \emph{quasilinear
  modified energy method} as a quasilinear alternative to the normal
form analysis (for further details see \cite{HIT}).

In addition to the above energy estimates for the full system, another
key role in the present paper is played by the cubic estimates for the
linearization of the original system~\ref{ww2d1}.  To introduce these
linearized (exactly as in \cite{HIT}), we denote the linearized
variables by $(w,q)$. However, for the analysis it is more convenient
to work with the ``good variables'' $(w,r)$ where
\[
r = q- Rw.
\]
Expressed in terms of $(w,r)$, the linearized equations take the form
\begin{equation}\label{ww2d-lin}
\left\{
\begin{aligned}
& (\partial_t + b \partial_\alpha) w  +  \frac{1}{1+\bar \W} r_\alpha
+  \frac{R_{\alpha} }{1+\bar \W} w  = \mathcal{G}(w,r),
 \\
&(\partial_t + b \partial_\alpha)  r  - i  \frac{1+a}{1+\W} w  = \mathcal{K}(w,r),
\end{aligned}
\right.
\end{equation}
where the functions $b$ and $a$ are as in \eqref{defa} and \eqref{defb} and
\begin{equation*}
\begin{aligned}
\mathcal{G}(w,r) = \ (1+\W) (P \bar m + \bar P  m), \quad \mathcal{K}(w,r) =  \  \bar P n - P \bar n.
\end{aligned}
\end{equation*}
We also recall the definitions of $m$ and $n$: 
 \[ 
 m := \frac{r_\alpha +R_\alpha w}{J} + \frac{\bar R w_\alpha}{(1+\W)^2}, \qquad n := \frac{ \bar R(r_{\alpha}+R_\alpha w)}{1+\W}.
\]

In particular, we remark that the linearization of the system \eqref{ww2d-diff} around the zero solution is 
\begin{equation} \label{ww2d-0} \left\{
\begin{aligned}
 & w_{ t} +  r_\alpha   =  0,
\\
& r_t - i w = 0.
\end{aligned}
\right.
\end{equation}
This system is a well-posed linear evolution in the space $\H$ of holomorphic functions,
and a conserved energy for this system is
 \begin{equation}\label{E0}
 \Ez (w,r) = \int \frac12 |w|^2 + \frac{1}{2i} (r \bar r_\alpha - \bar r r_\alpha) d\alpha
\approx \| (w,r)\|_{\H}^2
 \end{equation}


We remark that while $(w,r)$ are holomorphic, it is not directly obvious
that the  evolution \eqref{ww2d-lin} preserves the space of holomorphic states. To
remedy this one can also project the linearized equations onto the
space of holomorphic functions via the projection $P$.  Then we obtain
the equations
\begin{equation}\label{lin(wr)}
\left\{
\begin{aligned}
& w_t + P  \left[ b \partial_\alpha w \right]  + P \left[ \frac{1}{1+\bar \W} r_\alpha\right]
+  P \left[ \frac{R_{\alpha} }{1+\bar \W} w \right] = P \mathcal{G}(  w, r),
 \\
& r_t + P \left[b \partial_\alpha   r\right]  - i P\left[ \frac{1+a}{1+\W} w\right]  =
 P \mathcal{K}( w,r).
\end{aligned}
\right.
\end{equation}
Since the original set of equations \eqref{ww2d1} is fully
holomorphic, it follows that the two sets of equations,
\eqref{ww2d-lin} and \eqref{lin(wr)}, are algebraically equivalent.

For this problem we add appropriate cubic terms to the linear energy functional 
$\Ez$ above and define the quasilinear modified energy
\[
\Elint(w,r) := \int_{\R} (1+a) |w|^2 + \Im (r \bar  r_\alpha)
+ 2 \Im (\bar R w r_\alpha) -2\Re(\bar{\W} w^2)\, d\alpha.
\]

Then we have the following cubic energy estimate, which shows that the system 
\eqref{ww2d-lin} is well-posed in $\H$:
\begin{theorem}
\label{t:lin}
Assume that $A \ll1$. Then the energy functional $\Elint(w,r)$ has the following properties
as long as $A \ll 1$:

(i) Norm equivalence:
\begin{equation}\label{elin3-eq}
 \Elint(w,r) = (1+O(A)) \|(w,r)\|_{\H}^2.
\end{equation}

(ii) The solutions to \eqref{ww2d-lin} satisfy
\begin{equation}\label{elin3-diff}
\begin{split}
  \left|\frac{d}{dt} \Elint(w,r) \right| \lesssim  AB  \|(w,r)\|_{\H}^2
\end{split}
\end{equation}
\end{theorem}

We remark that, due to the invariance of the system \eqref{ww2d1} with
respect to spatial translations, a particular solution $(w,r)$ for the
linearized system is given by $(\W,R)$.

\section{ Normal forms and the approximate solution}

As observed in \cite{HIT}, the quadratic terms in the water wave
equation can be eliminated via a quadratic normal form
transformation.  Here we will develop this a step further, eliminating also the 
cubic non-resonant terms with an additional cubic correction to the normal form
transformation. This will then be formally inverted, in order to produce 
an approximate water wave solution from the approximate cubic NLS
solution $\tY^\epsilon$.

We begin with the cubic expansion in the water wave system \eqref{ww2d1}. To compute that we have
\[
\left(\frac{1}J \right)^{(\leq 3)} = 1 - W_\alpha - \bar W_\alpha - |W_\alpha|^2  + (W_\alpha + \bar W_\alpha)^2 ;
\]
therefore
\[
\begin{split}
F^{(\leq 3)} =  & \  Q_\alpha - Q_\alpha W_\alpha - P [ Q_\alpha \bar W_\alpha - \bar Q_\alpha W_\alpha] + 
P [ (Q_\alpha - \bar Q_\alpha) (W_\alpha^2 + |W_\alpha|^2 + \bar W_\alpha^2)] 
\\
= & \  Q_\alpha - Q_\alpha W_\alpha - P [ Q_\alpha \bar W_\alpha - \bar Q_\alpha W_\alpha] + 
Q_\alpha W_\alpha^2 + 
P [ Q_\alpha  (|W_\alpha|^2 + \bar W_\alpha^2)] \\ & \  -  P [ \bar Q_\alpha(W_\alpha^2 + |W_\alpha|^2 )].
\end{split}
\]
Here and later, the superscript $(\leq 3)$ denotes terms up to cubic order in a formal power series 
expansion in $(W,Q)$. Using this in \eqref{ww2d1} we get
\begin{equation}
\label{ww2d1-re}
\left\{
\begin{aligned}
& W_t + Q_\alpha = G \\
& Q_t -i W    = K, \\
\end{aligned}
\right.
\end{equation}
where the quadratic and cubic terms in $G$ and $K$ are given by 
\[
\left\{
\begin{aligned}
G^{(\leq 3)} := & \    (1+W_\alpha) P [ Q_\alpha \bar W_\alpha - \bar Q_\alpha W_\alpha] 
- P [ Q_\alpha  (|W_\alpha|^2 + \bar W_\alpha^2)] +  P [ \bar Q_\alpha(W_\alpha^2 + |W_\alpha|^2 )]
\\
K^{\leq 3)}  := & \ - Q_\alpha^2 - P[|Q_\alpha|^2] + W_\alpha Q_\alpha^2 + Q_\alpha P [ Q_\alpha \bar W_\alpha - \bar Q_\alpha W_\alpha] + P[ |Q_\alpha|^2(W_\alpha + \bar W_\alpha)].
\end{aligned}
\right.
\]
As observed in \cite{HIT}, the quadratic terms in $G$ and $K$ can be eliminated 
using a normal form transformation
\begin{equation}
\tilde W = W - 2 P[\Re W  W_\alpha], \qquad \tilde Q = Q - 2 P[ \Re W  Q_\alpha]
\label{nft1}
\end{equation}
Of the cubic terms, the ones with exactly one complex conjugate are resonant, 
while the ones with either two conjugates or with none are always non-resonant.
For our purposes here, it is convenient to add a further cubic correction to 
our normal form transformation which eliminates these non-resonant terms.
This is given by

\begin{equation}
\left\{
\begin{aligned}
\tilde W = & \  W - 2 P[\Re W  W_\alpha] 
+ \frac12  \partial_\alpha ( W^2 W_\alpha) + \frac12 \partial_\alpha P[ \bar W^2 W_\alpha]   
\\
\tilde Q = & \ Q - 2 P[ \Re W  Q_\alpha]+ \frac12  \partial_\alpha ( W^2 Q_\alpha) 
+ \frac12 \partial_\alpha P[ \bar W^2 Q_\alpha]   .
\end{aligned}
\right.
\label{nft3}
\end{equation}
This will not be directly used here except in order to
gain intuition and to motivate our construction of the approximate
solution, which formally inverts this normal form transformation.
With this transformation, the equations for $(\tilde W,\tilde Q)$ become
\begin{equation}
\left\{
\begin{aligned}
&\tW_t + \tQ_\alpha = \tG,
\\
&\tQ_t - i \tW =  \tK,
\end{aligned}
\right.
\label{nft-ww2d1}
\end{equation}
where the quadratic term in $\tG$, $\tK$  vanishes and the cubic term has the form
\begin{equation}\label{tGK3}
\left\{
\begin{aligned}
\tG^{(3)}  =  &  \partial_\alpha \left[ W P[W_\alpha \bar Q_\alpha - 
Q_\alpha \bar W_\alpha]\right]
\\
\tK^{(3)} = & \   W P[[Q_\alpha|^2]_\alpha 
+ Q_\alpha P[ \bar Q_\alpha W_\alpha - \bar W_\alpha Q_\alpha] + \partial_{\alpha}P\left[ Q^2_{\alpha}\bar{W}\right].
\end{aligned}
\right.
\end{equation}
We can  further split the above expressions into resonant terms and null terms, where the latter 
vanish when applied to three waves which travel in the same direction.
We summarize the 
result in the following
\begin{lemma}
The normal form transformation \eqref{nft3} yields an equation of the form \eqref{nft-ww2d1}
where the cubic terms in $\tG$ and $\tK$ are given by

\begin{equation*}
\left\{
\begin{aligned}
\tG^{(3)}_{r}=&   \ 0
\\
\tG^{(3)}_{null}=& \        \partial_\alpha \left[ W P[W_\alpha \bar Q_\alpha - 
Q_\alpha \bar W_\alpha]\right]
\\
\tK^{(3)} _{r}= &\ \partial_{\alpha}P\left[ Q^2_{\alpha}\bar{W}\right]                          \\
\tK^{(3)} _{null}=& \  W P[[Q_\alpha|^2]_\alpha 
+ Q_\alpha P[ \bar Q_\alpha W_\alpha - \bar W_\alpha Q_\alpha].
\end{aligned}
\right.
\end{equation*}
\end{lemma}

 Keeping only the 
cubic resonant terms, the equations for $(\tilde W,\tilde Q)$ can be formally written as
\begin{equation}
\left\{
\begin{aligned}
&\tW_t + \tQ_\alpha \approx 0,
\\
&\tQ_t - i \tW \approx  \partial_{\alpha}P\left[ Q^2_{\alpha}\bar{W}\right],
\end{aligned}
\right.
\label{nft1res}
\end{equation}
or as a linearly diagonal system
for the variables 
\[
Y^+ = \frac12(\tW  + |D|^\frac12 \tQ), \qquad Y^- = \frac12(\tW  - |D|^\frac12 \tQ)
\]
\begin{equation}
\left\{
\begin{aligned}
&(i \partial_t + |D|^\frac12) Y^+ \approx \frac{i}{2} \partial_\alpha \vert D\vert^{\frac{1}{2}}P[ \tQ^2_\alpha  \bar{ \tW}],
\\
&(i \partial_t - |D|^\frac12)  Y^- \approx -\frac{i}2 \partial_\alpha \vert D\vert^{\frac{1}{2}} P[ \tQ^2_\alpha \bar{ \tW}].
\end{aligned}
\right.
\label{nft1res-y}
\end{equation}

Now we further specialize to solutions which are localized near
frequency $\xi_0 = -1$ and along the branch $\tau =
\sqrt{|\xi|}$. Being localized along $\tau = \sqrt{|\xi|}$ corresponds
to discarding the second equation in \eqref{nft1res-y} and setting
$Y^-=0$ in the right hand side of the first equation in
\eqref{nft1res-y}, and then being localized near frequency $\xi_0 =
-1$ is akin to setting $Q = W$ and $\partial_\alpha = - i$, $|D|= 1$
in the right hand side of the first equation in \eqref{nft1res-y}.

Explicitly, we express $\tW$ and $\vert D\vert^{1/2}\tQ$ in terms of $Y^+$ and $Y^-$:
\[
\tW =Y^+ +Y^-, \qquad  \vert D\vert^{\frac{1}{2}}\tQ=Y^+-Y^-,
\]
substitute them in the right hand side of the first equation in \eqref{nft1res-y}, and then set $Y^- =0$ to arrive (still formally) at an approximate equation for $Y^+$,
namely
\[
(i \partial_t + |D|^\frac12) Y^+ \approx - \frac{1}{2}  Y^+ |Y^+|^2.
\]
This will motivate the choice $\lambda = -\frac12$ in \eqref{NLS}.
The above heuristics show that $Y^+$ is an approximate solution for
the cubic NLS equation \eqref{NLS-y}.  

Now we use the above heuristics in order to construct the approximate solution
$(W^\epsilon,Q^\epsilon)$.  We start with an approximate solution
$\tY^\epsilon$ for the cubic NLS problem, and we seek to recover an
approximate solution $(\tW,\tQ)$ for the water wave equation. We start
from the choice suggested by the above heuristics, namely
\begin{equation}\label{chooseY+}
Y^+ =  \tY^\epsilon.
\end{equation}

One might now consider setting $Y^- = 0$, but this would correspond to
neglecting the cubic $Y^+$ part in the right hand side in the $Y^-$
equation \eqref{nft1res-y}.  Instead, we note that if we harmlessly neglect the $Y^-$
part of the right hand of the $Y^-$ equation then we are left with a
cubic expression in $Y^+$ which is non-resonant (i.e., this expression will not  cancel on the branch $\tau =-\sqrt{\xi}$).  Then we compute the
normal form correction
\begin{equation}\label{chooseY-}
Y^- =  -\frac{1}{4} \tY^\epsilon |\tY^\epsilon|^2 .
\end{equation}

This formally leads us to the choice of $\tW^\epsilon$ and $\tQ^\epsilon$ which solve the normal
form water wave equation \eqref{nft1res}   with quartic accuracy,
\begin{equation}\label{YtoWQ}
\tW^\epsilon =  \tY^\epsilon -  \frac{1}{4} \tY^\epsilon |\tY^\epsilon|^2  \qquad \tQ^\epsilon = |D|^{-\frac12} (  \tY^\epsilon +  \frac{1}{4} \tY^\epsilon |\tY^\epsilon|^2).
\end{equation}
Next we formally invert the normal form transformation \eqref{nft3}
to cubic order, setting
\begin{equation}\label{nft1+a}
\left\{
\begin{aligned}
W^\epsilon = & \  \tW^\epsilon + 2 P[\Re \tW^\epsilon  \tW^\epsilon_\alpha]  + 4 P[\Re  P[\Re \tW^\epsilon  \tW^\epsilon_\alpha]  \tW^\epsilon_\alpha]  
+ 4 P[\Re \tW^\epsilon P[\Re \tW^\epsilon  \tW^\epsilon_\alpha]_\alpha]
\\ & \ 
- \frac12  \partial_\alpha ( (\tW^\epsilon)^2 \tW^{\epsilon}_\alpha) - \frac12 \partial_\alpha P[ (\bar \tW^\epsilon)^2 \tW^\epsilon_\alpha] 
\\
 Q^\epsilon = & \ \tQ^\epsilon + 2 P[ \Re \tW^\epsilon \tQ^\epsilon_\alpha] + 4 P[\Re  P[\Re \tW^\epsilon  \tW^\epsilon_\alpha]  \tQ^\epsilon_\alpha]  
+ 4 P[\Re \tW^\epsilon P[\Re \tW^\epsilon  \tQ^\epsilon_\alpha]_\alpha]
\\ & \   -\frac12  \partial_\alpha ( (\tW^\epsilon)^2 \tQ^{\epsilon}_\alpha) 
- \frac12 \partial_\alpha P[ (\bar \tW^\epsilon)^2 \tQ^{\epsilon}_\alpha]  .
\end{aligned}
\right.
\end{equation}
This will be our candidate for an approximate solution to \eqref{ww2d1}.

Proceeding in two steps, we first  show that $(\tW^\epsilon,\tQ^\epsilon)$ is a good approximate 
solution to the cubic system
\begin{equation}
\left\{
\begin{aligned}
&\tW^\epsilon_t + \tQ^\epsilon_\alpha = \tG^{(3)}(\tW^\epsilon,\tQ^\epsilon) + \tg^{\epsilon}
\\
&\tQ^\epsilon_t - i \tW^\epsilon =  \tK^{(3)}(\tW^\epsilon,\tQ^\epsilon) + \tk^{\epsilon},
\end{aligned}
\right.
\label{nft2-eps}
\end{equation}
where $\tG^{(3)}$ and $\tK^{(3)}$ are given by \eqref{tGK3}.

\begin{lemma}
Let $(\tW^\epsilon,\tQ^\epsilon)$ be defined by \eqref{YtoWQ}.
Then they have the following properties:
\begin{description}
\item[(i)] Frequency localization at frequency $\xi \in \left(-2, 0\right)$.

\item[(ii)] Close to $\tY^\epsilon$, 
\begin{equation}
\| \tW^\epsilon - \tY^\epsilon\|_{L^2} + \| \tQ^\epsilon - \tY^\epsilon\|_{L^2} \lesssim \epsilon^{\frac32}.
\end{equation}

\item[(iii)] Approximate solution,
\begin{equation}\label{tgk}
\| \tg^{\epsilon}\|_{L^2} +  \| \tk^{\epsilon}\|_{L^2} \lesssim \epsilon^{\frac72}.
\end{equation}

\end{description}
\end{lemma}
\begin{proof}
Part (i) is trivial. For part (ii) we note that by Sobolev embeddings we have 
\begin{equation}\label{Y-bd}
\| \tY^\epsilon\|_{L^2} \lesssim \epsilon^\frac12, \qquad \| \tY^\epsilon\|_{L^\infty} \lesssim \epsilon,
\end{equation}
which suffices for the cubic terms. For the linear term on the other hand we use 
\begin{equation}\label{Y-at-i}
\| (D+1) Y_\epsilon \|_{L^2}  \lesssim \epsilon^\frac32.
\end{equation}

We now turn our attention to part (iii).  By the bounds in \eqref{Y-bd} we can freely replace 
$\tW^\epsilon$ and $\tQ^\epsilon$ by $\tY^\epsilon$ in $\tG^{(3)}$ and $\tK^{(3)}$.
Using both  \eqref{Y-bd} and \eqref{Y-at-i} we can also estimate directly the null parts
\[
\| \tG^{(3))}_{null}\|_{L^2} +  \| \tK_{null}^{(3)}\|_{L^2} \lesssim \epsilon^{\frac72}.
\]

Next we consider the time derivatives of $\tW^\epsilon$ and $\tQ^\epsilon$.
For $\tW^\epsilon$ we have 
\[
\partial_t \tW^\epsilon = \partial_t \tY^\epsilon - \frac14( 2  \partial_t \tY^\epsilon  |\tY^\epsilon|^2 +
 \partial_t \bar{\tY}^\epsilon  (\tY^\epsilon)^2) .
\]
For $\tY^{\epsilon}$ we use \eqref{eqn-ye}, which gives
\[
\partial_t   \tY^\epsilon= i \omega_0(D) )\tY^\epsilon 
+\frac{i}{2} \tY^\epsilon |\tY^\epsilon|^2-ig^\epsilon.
\]
The $g^{\epsilon}$ is acceptable by \eqref{ge-bound}.
Using again \eqref{Y-bd} we can discard the quintic terms arising from $\partial_t \tY^\epsilon$,
and using also \eqref{Y-at-i} we can replace $i \omega_0(D)  \tY^\epsilon$ by $i  \tY^\epsilon$ 
in the cubic terms, to obtain
\[
\partial_t \tW^\epsilon = i \omega_0(D)  \tY^\epsilon 
+\frac{i}4   \tY^\epsilon  |\tY^\epsilon|^2 + 
O_{L^2}(\epsilon^{\frac72}).
\]
Similarly we have 
\[
\partial_t \tQ^\epsilon = i |D|^{-\frac12} \omega_0(D)  \tY^\epsilon + \frac{3i}4   \tY^\epsilon  |\tY^\epsilon|^2 + 
O_{L^2}(\epsilon^{\frac72}),
\]
as well as 
\[
\partial_\alpha  \tQ^\epsilon  = - i |D|^{\frac{1}{2}}  \tY^\epsilon - \frac{i}4   \tY^\epsilon  |\tY^\epsilon|^2 + 
O_{L^2}(\epsilon^{\frac72}).
\]

Thus for the first equation in \eqref{nft2-eps} we compute
\[
\begin{split}
e_1 := & \ \partial_t \tW^\epsilon + \partial_\alpha  \tQ^\epsilon - \tG_r^{(3)}(\tY^\epsilon,\tY^\epsilon) 
\\
= & \  i (\omega_0(D) - |D|^\frac12)  \tY^\epsilon+ 
O_{L^2}(\epsilon^{\frac72}).
\end{split}
 \]
Since $\omega_0(D)$ agrees to cubic order  with $|D|^\frac12$ at $\xi = -1$, 
by \eqref{ye-bound} we have
\[
\|  (\omega_0(D) - |D|^\frac12)  \tY^\epsilon\|_{L^2} \lesssim \|(D+1)^3  \tY^\epsilon\|_{L^2} 
\lesssim \epsilon^{\frac72}
\]
as needed. The computation for the second equation  in \eqref{nft2-eps} is similar.
\end{proof}

The second step is to make the transition from $(\tW^\epsilon,\tQ^\epsilon)$ to 
$(W^\epsilon,Q^\epsilon)$, which will solve an equation of the form 
\begin{equation}
\label{ww2d1-eps+}
\left\{
\begin{aligned}
& W^\epsilon_t + Q^\epsilon_\alpha = G^{\leq 3}(W^\epsilon, Q^\epsilon) + g^\epsilon  \\
& Q_t -i W    = K^{\leq 3}(W^\epsilon, Q^\epsilon) + k^\epsilon . \\
\end{aligned}
\right.
\end{equation}
Precisely, we have:

\begin{lemma}
Let $(W^\epsilon,Q^\epsilon)$ be defined by \eqref{nft1+a}. 
Then they have the following properties:
\begin{description}
\item[(i)] Frequency localization at frequency $\xi \in (-6, 0)$.

\item[(ii)] Close to $\tY^\epsilon$, 
\begin{equation}
\| W^\epsilon - \tY^\epsilon\|_{L^2} + \| Q^\epsilon - \tY^\epsilon\|_{L^2} \lesssim \epsilon^{\frac32}.
\end{equation}

\item[(iii)] Approximate solution,
\begin{equation}
\| g^\epsilon\|_{L^2} +  \| k^{\epsilon}\|_{L^2} \lesssim \epsilon^{\frac72}.
\end{equation}

\end{description}
\end{lemma}

\begin{proof}
Since the reverse normal form transformation \eqref{nft1+a} formally inverts 
the direct normal form transformation \eqref{nft3}, it follows that all the terms up to cubic order 
in $(\tW^\epsilon,\tQ^\epsilon)$ will cancel in \eqref{ww2d1-eps+}. Hence the errors
$(g^\epsilon,k^\epsilon)$ will contain two types of terms:

\bigskip

a) Quartic and higher order terms in $(\tW^\epsilon,\tQ^\epsilon)$. By
part (ii) of the previous Lemma, the functions
$(\tW^\epsilon,\tQ^\epsilon)$ will also satisfy \eqref{Y-bd}, which
allows us to estimate these terms.

b) Bilinear and higher order terms in $(\tW^\epsilon,\tQ^\epsilon)$ and
$(\tg^\epsilon,\tk^\epsilon)$ which are linear in the latter. There we combine 
\eqref{Y-bd} for $(\tW^\epsilon,\tQ^\epsilon)$ with \eqref{tgk}.
\end{proof}

To finish the proof of Theorem~\ref{t:approx} it remains to estimate the quartic and higher order terms
in the nonlinear terms $(G,K)$ in \eqref{ww2d1-re},
\[
\| (G^{(\geq 4)},K^{(\geq 4)}) (W^\epsilon,Q^\epsilon)\|_{H^N} \lesssim \epsilon^\frac72.
\]
Here we are also asking for higher Sobolev regularity since these
expressions are fully nonlinear, in that they may also contain factors
of $\dfrac{1}{1+W_\alpha}$ and its conjugate.  For the $L^2$ bound we
simply use the analogue of \eqref{Y-bd} for $(W^\epsilon,Q^\epsilon)$
and bound the nonlinear factors as above in $L^\infty$. For higher
regularity we differentiate, distribute derivatives and repeat the
same argument.


\section{ Energy estimates and the exact solution}
\label{s:exact}

To compare the approximate solution $(W^\epsilon,Q^\epsilon)$ given by Theorem~\ref{t:approx}
with the exact solution $(W,Q)$ with the same data at time $0$ we consider, within the time interval 
\[
I = [0,T/\epsilon^2],
\]
a one parameter family of solutions $(W(t,s),Q(t,s))$ to \eqref{ww2d1}. 
Here $t$ represents time and $s$ is the  parameter. These solutions are defined as the solutions to the 
Cauchy problem for the water wave equations \eqref{ww2d1} with initial data at time $t = s$,
\[
(W(s,s), Q(s,s)) = (W^\epsilon(s),Q^\epsilon(s)).
\]

A-priori these solutions exist on a time interval around $t = s$, but
not necessarily up to time $T \epsilon^{-2}$.  We claim that if
$\epsilon$ is small enough then indeed they exist up to this time and
further that their control parameters satisfy the uniform bounds
\[
A(t,s),B(t,s) \leq M\epsilon. 
\]
for a universal constant $M$.

To prove our claim we use a continuity argument, based on the fact
that $A$ and $B$ are continuous functions of $t,s$ wherever they are
defined.  Precisely, we make the bootstrap assumption
\[
A(t,s), B(t,s) \leq 2M\epsilon, \qquad s_0 \leq s \leq t \leq T \epsilon^{-2},
\]
and show that for $M$ large enough the constant $2M$  can be improved to $M$. This will show that 
the minimal $s_0$ for which the above property holds is indeed $s_0 = 0$.

We begin our analysis by estimating high frequencies, which is achieved by propagating regularity in $t$, 
applying the energy estimates in Theorem~\ref{t:energy}. These show that the 
solutions $(W(t,s),Q(t,s))$ satisfy the uniform bound
\begin{equation}\label{e1}
\| (\W,R)(t,s) \|_{H^N} \lesssim e^{4CM^2 T} \epsilon^\frac12.
\end{equation}

On the other hand, in order to estimate the low frequencies we
consider the $s$ derivatives of $W$ and $Q$, namely
$(w,q) = (\partial_s W(t,s),\partial_sQ(t,s))$. These solve the
linearized equation \eqref{ww2d-lin} in $t$, with initial data
\[
(w,q)(s,s) = (\tG(s),\tK(s)).
\]
We switch to the good variable $r = q- Rw$, and apply Theorem~\ref{t:lin}
to obtain the uniform bound
\begin{equation}\label{e2}
\|(w,r)(t,s) \|_{\H} \lesssim e^{4CM^2 T} \epsilon^{\frac{7}2}.
\end{equation}

Integrating in $s$, the $w$ bound yields
\begin{equation}\label{e3}
\| W(t,s) - W(t,t)\|_{L^2} \lesssim  T e^{4cM^2 T} \epsilon^{\frac{3}2}.
\end{equation}
Similarly, combining the $w$ and $r$ bounds via 
\[
Q_s = r + Rw,
\]
we get 
\begin{equation}\label{e3+}
\| Q(t,s) - Q(t,t)\|_{\dot H^\frac12+\epsilon L^2} \lesssim  T e^{4CM^2 T} \epsilon^{\frac{3}2}.
\end{equation}

Next we consider $R$, for which we have 
\[
R_s = \frac{Q_{\alpha s}- R W_{\alpha s}}{1+W_\alpha}  = \frac{ r_\alpha  + R_\alpha w}{1+W_\alpha}   
\]
We can bound this in $H^{-\frac12}$ as 
\[
\| R_s \|_{H^{-\frac12}} \lesssim  e^{8CM^2 T} \epsilon^{\frac{7}2}.
\]
Then we can integrate in $s$ to  estimate
 \begin{equation}\label{e4}
\| R(t,s) - R(t,t)\|_{H^{-\frac12}} \lesssim  T e^{8CM^2 T} \epsilon^{\frac{3}2}.
\end{equation}

Combining now the bounds \eqref{e1}, \eqref{e3}, \eqref{e4} and interpolating we obtain
\begin{equation}
\| (\W(t,s) - \W(t,t), R(t,s) - R(t,t))\|_{\H^N} \lesssim T e^{8CM^2 T} \epsilon^{\frac{3}2-\delta}.
\end{equation}
Then by Sobolev embeddings
\[
A(t,s) + B(t,s) \lesssim  e^{8cM^2 C} \epsilon^{\frac{3}2-\delta} + A(t,t) + B(t,t) \lesssim  
Te^{8CM^2 T} \epsilon^{\frac{3}2-\delta}  + \epsilon
\]
Here the implicit constant does not depend on $M$. On the other hand $M$ does appear in the 
first exponential, but there it is controlled by the extra power of $\epsilon$ provided 
that $\epsilon$ is small enough.
This concludes our bootstrap argument. 

This does not quite conclude the proof of Theorem~\ref{t:NLS-approx}, as so far, instead of 
\eqref{error-hi} we only get 
\begin{equation}\label{almost-hi}
\|(\W - i Y^\epsilon,R - i Y^\epsilon )\|_{\H^N} \lesssim T e^{8CM^2 T} \epsilon^{\frac{3}2-\delta}. 
\end{equation}
We will rectify this in the next section.

\section{Further energy estimates}

\subsection{ Proof of Theorem~\ref{t:NLS-approx}, completed}

Here we will prove a better high frequency bound for $(W,Q)$, namely 
\begin{equation}\label{hh}
\| (D+1)(W,Q)\|_{\H^N} \lesssim T e^{8cM^2 T} \epsilon^{\frac{3}2}. 
\end{equation} 
By \eqref{e3} and \eqref{e3+} we already have the low frequency part of this bound, 
so it only remains to estimate the high frequencies. Since the initial data is frequency localized,
we have this bound at the initial time so it remains to propagate it up to time $T_\epsilon$.

In view of \eqref{almost-hi},
our starting point is the bound
\begin{equation}\label{hh-}
\| (D+1)(W,Q)\|_{\H^N} \lesssim T e^{8CM^2 T} \epsilon^{\frac{3}2-\delta}. 
\end{equation}

To $(W,Q)$ we apply the normal form transformation \eqref{nft3} to obtain the associated 
normal form variables
$(\tW,\tQ)$. By \eqref{hh-}, it suffices to prove \eqref{hh} for $(\tW,\tQ)$. 
On the other hand, we also have \eqref{hh-} for $(\tW,\tQ)$. 

The normal form variables $(\tW,\tQ)$ solve the equation \eqref{nft-ww2d1}. We separate $(\tG,\tK)$ into cubic 
and quartic and higher terms. Using \eqref{hh-} and the $\epsilon$ uniform bound for $A$ and $B$,
 the quartic terms are estimated by 
\[
\| (\tG^{\geq 4},\tK^{\geq 4})\|_{\H^N} \lesssim T e^{8CM^2 T} \epsilon^{\frac{7}2}, 
\]
where the worst contribution is the low frequency part where we can only use the $\epsilon^\frac12$ 
bound in $L^2$. Thus we have 
\begin{equation}
\left\{
\begin{aligned}
&\tW_t + \tQ_\alpha = \tG^{(3)} + O_{H^N} (e^{8cM^2 T} \epsilon^{\frac{7}2}) 
\\
&\tQ_t - i \tW =  \tK^{(3)} + O_{H^N} (e^{8cM^2 T} \epsilon^{\frac{7}2}) 
\end{aligned}
\right.
\label{nft-ww2d-high}
\end{equation}
where $(\tG^{3},\tK^{3})$ are given by \eqref{tGK3}.

The key feature of the cubic nonlinearity is that all terms have exactly one complex conjugate, i.e. 
the nonlinearity has a phase shift invariance. Because of this, when we apply the operator 
$D+1$ to the above equation we can distribute it to each of the factors in $(\tG^{3},\tK^{3})$.
Denoting 
\[
(w,q) = (D+1)(\tW,\tQ)
\]
we can thus write an equation for $(w,q)$. Before writing it, we simplify it as follows:

\begin{itemize}
\item We replace $(D+1)(W,Q)$ by $(D+1)(\tW,\tQ)$ and thus by $(w,q)$ in $(\tG^{3},\tK^{3})$,
at the expense of quartic terms which can be placed into the error as above.

\item We localize the remaining $(W,Q)$ factors in $(\tG^{3},\tK^{3})$ to frequencies $\lesssim 1$,
as the terms with high frequencies can be placed into the error using \eqref{hh-}.

\item We discard all $(\bar w,\bar q)$ terms, for they either have high frequency and then they are
eliminated by the projection, or they have low frequency and then their contributions can be placed
into the error directly.

\item We further replace the low frequency factors by $\tY^\epsilon$, and similarly their derivatives 
by $-i Y^\epsilon$, etc, using \eqref{almost-hi}. Here the terms involving $\partial_\alpha |\tY^\epsilon|^2$
will be only partially discarded for symmetry purposes.
\end{itemize}

We recall that 
\[
\left\{
\begin{aligned}
\tG^{(3)}  =  &  \partial_\alpha \left[ W P[W_\alpha \bar Q_\alpha - 
Q_\alpha \bar W_\alpha]\right]
\\
\tK^{(3)} = & \   W P[[Q_\alpha|^2]_\alpha 
+ Q_\alpha P[ \bar Q_\alpha W_\alpha - \bar W_\alpha Q_\alpha] + \partial_{\alpha}P\left[ Q^2_{\alpha}\bar{W}\right].
\end{aligned}
\right.
\]
Then we are left with the following equation for $(w,q)$ :
\begin{equation}
\left\{
\begin{aligned}
w_t + q_\alpha =& \  i \partial_\alpha (\tY^\epsilon P[ \bar \tY^\epsilon w_\alpha])
  - i \partial_\alpha(\tY^\epsilon P[\bar \tY^\epsilon q_{\alpha}]) + O_{H^N} (T e^{8CM^2 T} \epsilon^{\frac{7}2}) 
\\
q_t - i w =  & \  \tY^\epsilon P[\bar \tY^\epsilon w_\alpha]
+ i \tY^\epsilon P[ \bar \tY^\epsilon q_\alpha]_{\alpha} \! -  \!  \tY^\epsilon P[ \bar \tY^\epsilon q_\alpha]
 \\ & \ -  i [  \partial_\alpha ( |\tY^\epsilon|^2 q_\alpha)
+  |\tY^\epsilon|^2 q_{\alpha\alpha}] \!   +   O_{H^N} (T e^{8CM^2 T} \epsilon^{\frac{7}2}) 
\end{aligned}
\right.
\label{nft1eq-diff}
\end{equation}
Here we can directly compute  energy estimates for
$(w,q)$ to get
\[
\frac{d}{dt} \| (w,q)\|_{\H}^2  \lesssim T e^{8CM^2 T} \epsilon^{\frac{7}2} \| (w,q)\|_{\H},
\]
and conclude that 
\[
 \| (w,q)\|_{\H} \lesssim T^2 e^{8CM^2 T} \epsilon^\frac32,
\]
as desired.  The same applies to higher order derivatives of $(w,q)$
because when we differentiate the last equation, all terms where the
derivative falls on $ \tY^\epsilon $ are matched by terms where the
derivative falls on $ \bar \tY^\epsilon $, canceling at the leading
order.  Thus we gain another $\epsilon$ factor and can be placed into
the error.

\subsection{ Proof of Theorem~\ref{t:NLS-approx+}}

We will prove the result in two steps.
\medskip

\textbf{STEP 1: Smooth data.}  First, assume that the initial data $(W^1_0,Q^1_0)$ is frequency 
localized at frequency $\lesssim 1$. We consider a one parameter family of data 
$(W^h_0,Q^h_0)_{h \in [0,1]}$
interpolating linearly between the data $(W_0,Q_0)$ given by Theorem~\ref{t:NLS-approx}.

A-priori these solutions are smooth, exist locally in time near $t=0$, and depend smoothly on $h$. 
We claim that they exist uniformly up to time $T_\epsilon$, and that their associated control 
norms satisfy the bound
\[
A(h,t) + B(h,t) \lesssim \epsilon.
\]
By a continuity argument, to prove this it suffices to prove that the above bound holds 
in an interval $[0,t_0]$  with $t_0 \leq T_\epsilon$ under a bootstrap assumption
\[
A(h,t) + B(h,t) \lesssim M \epsilon.
\]
On one hand, by Theorem~\ref{t:energy},  the bootstrap assumption will  insure the uniform bound
\begin{equation}\label{h1}
\| (\W^h,R^h)\|_{\H^N} \lesssim e^{CM^2 T} \epsilon^\frac12.
\end{equation}
On the other hand by Theorem~\ref{t:lin} we get
for the linearized variables
\[
(w,q) = \partial_h (W,Q), \qquad r = q - Rw
\]
the bound
\[
\|  (w,r)\|_{\H} \lesssim e^{4CM^2 T} \epsilon^{1+\delta}.
\]
 
Arguing as in the proof of \eqref{e3}, \eqref{e3+} and \eqref{e4}, by integrating in $h$  we obtain
\begin{equation}\label{h3}
\| W^1-W\|_{L^2} \lesssim  e^{4CM^2 T} \epsilon^{1+\delta},
\end{equation}
\begin{equation}\label{h3+}
\| Q^1 - Q\|_{\dot H^\frac12+\epsilon L^2} \lesssim  e^{4CM^2 T} \epsilon^{1+\delta},
\end{equation}
 \begin{equation}\label{h4}
\| R^1 - R\|_{H^{-\frac12}} \lesssim  e^{4CM^2 T}\epsilon^{1+\delta}.
\end{equation}
After interpolation with \eqref{h1} this leads to 
\begin{equation}\label{h1+}
\| (\W^1-\W,R^1-R)\|_{\H^N} \lesssim e^{4CM^2 T} \epsilon^{1+\delta_1}. \qquad 0 < \delta_1 < \delta.
\end{equation}
Combining the last bound \eqref{h1+}, and Sobolev embeddings we obtain
\[
A(t,s) + B(t,s) \lesssim  \epsilon + e^{4CM^2 T}  \epsilon^{1+\delta_1}  \lesssim   \epsilon
\]
for small enough $\epsilon$ (depending on $T$) and conclude the bootstrap.

Now the same argument as in the previous subsection yields the bound
\begin{equation}
\| (D+1)(W^1,Q^1)\|_{\H^N} \lesssim  e^{4CM^2 T} \epsilon^{1+\delta}.
\end{equation}

\textbf{STEP 2. Rough data:} Here we borrow an idea from \cite{HIT}, Section 4.5,  namely to construct 
the solution $(W^1,Q^1)$  by starting with a regular  solution obtained by some regularization of the data,
and then by adding frequency layers to the solution. 

As in \cite{HIT}, it is more convenient to perform the frequency
localization at the level of the differentiated equation. Precisely,
for $k \geq 1$ we consider the sequence of initial data $(\W^{1}_{0,
  \leq k},R^{1}_{0,\leq k})$ for the differentiated equation. These
also correspond to data $(W^{1}_{0, \leq k},Q^{1}_{0,\leq k})$ for
the undifferentiated equation, as explained in \cite{HIT}. Here $k$ is
viewed as a continuous parameter, in order for us to be able to use
bounds for the linearized equation.

 The smooth starting point for this construction is $(\W^{1}_{0, \leq 1},R^{1}_{0,\leq 1})$ and its 
corresponding undifferentiated data $(W^{1}_{0, \leq 1},Q^{1}_{0,\leq 1})$. We remark that 
while $W^{1}_{0, \leq 1} = P_{\leq 1} W^1_0$ is directly obtained by truncating $W^1_0$ in frequency,
this is no longer the case for $Q^{1}_{0,\leq 1}$, which is determined from the relation 
 \[
\partial_\alpha Q^{1}_{0,\leq 1} = R^1_{0,\leq 1} ( 1+ \W^{1}_{0, \leq 1}).
\]
This is still localized in frequency, but we need to be careful when inverting the derivative. 
This is not a problem because on the right we are multiplying holomorphic functions,
which moves frequencies away from zero.

For this starting point we need to verify that we can apply the bounds in STEP 1, i.e. that 
the bounds \eqref{error-lo} and \eqref{error-hi} hold. This is easy for all differences except for $Q^{1}_{0,\leq 1}$,
where we compute
\[
Q^{1}_{0,\leq 1} - \tY_0^\epsilon = P_{\leq 1} (Q^{1}_{0} - \tY^\epsilon_0) -  \partial_\alpha^{-1}
\left( P_{\leq 1} (R_0^1 \W_0^1) -  R^1_{0,\leq 1}  \W^{1}_{0, \leq 1}\right)
\]
In the last term after cancellations we are left only with output at frequency $1$ so the inverse derivative is harmless, 
and we can estimate it in $L^2$ by $\epsilon^{2+\delta}$ which is much better than needed.

 Then we consider the corresponding solutions  $(W^{1}_{0, \leq k},Q^{1}_{0,\leq k})$, which a-priori exist 
uniformly on a small time interval. To prove that they extend up to time $T_\epsilon$ we make the bootstrap assumption
\[
A(t,k) + B(t,k) \leq M \epsilon,
\]
and then prove a stronger bound
\begin{equation}\label{need}
A(t,k) + B(t,k) \lesssim \epsilon.
\end{equation}

Suppose that $\epsilon^{1+\delta} c_k$ is a frequency envelope for
$(\W,R)$ in $\H^1$. Then from Theorem~\ref{t:energy} we have the uniform energy bounds
\[
\| (\W^{1}_{\leq k},R^{1}_{\leq k})\|_{\H^N} \lesssim e^{4CM^2 T} 2^{(N-1)k} c_k \epsilon^{1+\delta} + \epsilon^\frac12, \qquad N \geq 2.
\]
On the other hand  from Theorem~\ref{t:lin} we have the difference bounds
\[
\| (\W^{1}_{0, \leq k+1}- \W^{1}_{0, \leq k},R^{1}_{0,\leq k+1} - R^{1}_{0,\leq k})\|_{\H^{-1}} \lesssim e^{4CM^2 T} 2^{-2k} c_k\epsilon^{1+\delta}, 
\]
see again \cite{HIT}, Section 4.5.

Interpolating we obtain
\[
\| (\W^{1}_{0, \leq k+1}- \W^{1}_{0, \leq k},R^{1}_{0,\leq k+1} - R^{1}_{0,\leq k})\|_{\H} \lesssim e^{4CM^2 T} 2^{-k} c_k\epsilon^{1+\delta_1}, 
\]
and
\[
\| (\W^{1}_{0, \leq k+1}- \W^{1}_{0, \leq k},R^{1}_{0,\leq k+1} - R^{1}_{0,\leq k})\|_{\H^2} \lesssim e^{4CM^2 T} 2^{k} c_k\epsilon^{1+\delta_1}. 
\]
This shows that the differences $(\W^{1}_{0, \leq k+1}- \W^{1}_{0, \leq k},R^{1}_{0,\leq k+1} - R^{1}_{0,\leq k})$
have $\H^1$ size $c_k$ and are localized at frequency $2^k$. The $\H^1$ bounds for both the truncated and 
the full solutions  $(\H^1, R^1)$ for the differentiated equation then follow after summation of the 
differences with respect to $k$,
\[
\| (\W^{1}_{\leq k}- \W^{1}_{\leq k},R^{1}_{\leq k} - R^{1}_{\leq k})\|_{\H^1} \lesssim e^{4CM^2 T}  \epsilon^{1+\delta_1} 
\]
This in turn yields the desired bound \eqref{need} and closes the bootstrap argument.

\bibliographystyle{plain}
\bibliography{refs}

\end{document}